\documentclass[12pt]{amsart}
\usepackage{amsmath}
\usepackage{amsfonts}
\usepackage{amssymb,color}
\usepackage{hyperref}
\usepackage{comment}

\hypersetup{
	colorlinks   = true, 
	urlcolor     = blue, 
	linkcolor    = blue, 
	citecolor   = red 
}

\newcommand{\R}{\mathbb{R}}

\newcommand{\C}{\mathbb{C}}
\newcommand{\Z}{\mathbb{Z}}

\newcommand{\F}{\mathcal{F}}
\newcommand{\E}{\mathbb{E}}
\renewcommand{\P}{\mathbb{P}}

\newcommand{\norm}[2]{\left\| #2 \right\|_{#1}}
\newcommand{\norma}[2]{\big\| #2 \big\|_{#1}}

\definecolor{darkviolet}{rgb}{0.58,0,0.83} 

\newtheorem{lemma}{Lemma}[section]
\newtheorem{theorem}[lemma]{Theorem}

\newtheorem{corollary}[lemma]{Corollary}
\newtheorem{prop}[lemma]{Proposition}

\theoremstyle{definition}

\newtheorem{rem}[lemma]{Remark}
\numberwithin{equation}{section}

\author{Jos\'e Luis Romero}
\address{Faculty of Mathematics \\
	University of Vienna \\
	Oskar-Morgenstern-Platz 1 \\
	 1090 Vienna, Austria \\and
	Acoustics Research Institute\\ Austrian Academy of Sciences\\Wohllebengasse 12-14, 1040 Vienna, Austria}
\email{jose.luis.romero@univie.ac.at}
\email{michael.speckbacher@univie.ac.at}

\author{Michael Speckbacher}

\thanks{J. L. R. and M. S. gratefully acknowledge support from the Austrian Science Fund (FWF): Y 1199 and J 4254.}

\title[Spectral-norm risk for multi-tapering]{Spectral-norm risk rates for multi-taper estimation of Gaussian processes}

\begin{document}
\begin{abstract}
We consider the estimation of the covariance of a stationary Gaussian process on a multi-dimensional grid from observations taken on a general acquisition domain. We derive spectral-norm risk rates for multi-taper estimators. When applied to one dimensional acquisition intervals, these show that Thomson's classical multi-taper has  optimal risk rates, as they match known benchmarks. We also extend existing lower risk bounds to multi-dimensional grids and conclude that multi-taper estimators associated with certain two-dimensional acquisition domains also have almost optimal risk rates.
\end{abstract}
\maketitle

\section{Introduction}
Let $X=\{X_k : k \in \mathbb{Z}^d\}$ be a stationary, real, zero mean, ergodic Gaussian process on the infinite $d$-dimensional grid $\mathbb{Z}^d$. The stochastics of $X$ are encoded in its covariance matrix
\begin{align}\label{eq_lll}
\Sigma_{n,m} = \mathbb{E} \big\{ X_n \cdot X_m \} = \sigma_{n-m}, \qquad n,m \in \mathbb{Z}^d,
\end{align}
or, equivalently, in its \emph{spectral density}
\begin{align*}
S(\xi) = \sum_{n \in \mathbb{Z}^d} \sigma_n e^{2 \pi i \langle\xi, n\rangle}, \qquad \xi \in \mathbb{R}^d.
\end{align*}
\emph{Single shot spectral estimation} is the task of approximating the covariance matrix $\Sigma$ given one realization of $X$ observed on an  \emph{acquisition domain} $\Omega \subset \mathbb{Z}^d$.

To compare the performance of different estimators, one must specify a suitable class of covariance matrices $\Sigma$ and an error metric for their approximations. In this article we consider processes $X$ whose spectral density $S$ is a twice continuously differentiable function on $\mathbb{R}^d$ and often fix the normalization
\begin{align}\label{eq_1}
\norm{C^2}{S} = \max \left\{
\max_{\xi \in \mathbb{R}^d} |S(\xi)|,
\max_{\xi \in \mathbb{R}^d} |\partial_{\xi_j} S(\xi)|,
\max_{\xi \in \mathbb{R}^d} |\partial_{\xi_j} \partial_{\xi_k} S(\xi)|:
j,k=1,\ldots,d \right\} \leq 1. 
\end{align}
The performance of an estimator $\widehat{S}$ will be measured with respect to the \emph{uniform norm}
\begin{align}\label{eq_a}
\big\|S - \widehat{S}\big\|_\infty=\max_{\xi \in \mathbb{R}^d} |S(\xi) - \widehat{S}(\xi)|,
\end{align}
as this measures the error incurred in approximating the covariance matrix $\Sigma$ in \emph{spectral norm}:
\begin{align*}
\big\|\Sigma - \widehat{\Sigma}\big\|_s := 
\max_{\norm{2}{a} \leq 1} \big\|\Sigma a - \widehat{\Sigma} a\big\|_2=\big\|S - \widehat{S}\big\|_\infty,
\end{align*}
where $\widehat{\Sigma}_{n,m} = \widehat{\sigma}_{n-m}$,
$\widehat{S}(\xi) = \sum_{n \in \mathbb{Z}^d} \widehat{\sigma}_n e^{2 \pi i \langle\xi, n\rangle}$, and $\|a\|_2 = (\sum_{n\in\mathbb{Z}^d} |a_n|^2)^{1/2}$.
\footnote{Indeed, by \eqref{eq_lll}, the covariance matrix $\Sigma$ represents a convolution operator on $\ell^2(\mathbb{Z}^d)$, which is unitarily equivalent via a Fourier expansion to a multiplication operator on $L^2([0,1]^d)$ with symbol $S$, see, e.g., \cite{MR0094840} for more background.}

The most classical setting for single shot spectral estimation concerns dimension $d=1$ and observations taken on a finite interval $\Omega=\{1, \ldots, N\}$. As shown in \cite[Theorem 1]{carezh13}, in this classical setup the \emph{minimax risk} satisfies
\begin{align}\label{eq_3}
\inf_{\widehat{S}} \sup_{S} \mathbb{E} \left\{ \big\|S - \widehat{S}\big\|_\infty^2 \right\}
\asymp \left(\frac{\log N}{N}\right)^{4/5},
\end{align}
where the supremum is taken over all stationary, real, zero mean, Gaussian processes $X$ with spectral density satisfying \eqref{eq_1} (which are necessarily ergodic) and the infimum runs over all estimators $\widehat{S}$ based on $X_1, \ldots, X_N$.\footnote{Here, and throughout, we write $f \lesssim g$ for two functions if there exists a constant $C>0$ such that $f(x) \leq C g(x)$ for all $x$, while $f \asymp g$ means $f \lesssim g$ and $g \lesssim f$.} Moreover, asymptotically optimal estimators approximately achieving the infimum in \eqref{eq_3} are known explicitly  \cite[Theorem 2]{carezh13}.

A first goal of this article is to show that Thomson's multitaper estimator \cite{th82}
\begin{equation}\label{eq:Smt}
\widehat{S}^{\mathrm{mt}}(\xi):=\frac{1}{K}\sum_{k=0}^{K-1}
\left|\sum_{n\in\Omega}X_n \cdot \nu^{(k)}_n(\Omega,W) \cdot e^{-2\pi i\langle\xi, n\rangle}\right|^2,
\end{equation}
where $\nu^{(k)}_n(\Omega,W)$ denote the Slepian sequences (see \cite{sle78} and  below), is  asymptotically optimal for the minimax risk \eqref{eq_3} when the corresponding parameters are chosen adequately. More generally, we investigate multi-taper estimators on general acquisition domains and prove upper and lower minimax risk estimates.

\section{Results}
\subsection{Multi-taper estimators}
While originally introduced to study time series \cite{th82}, we consider the multi-taper estimator for general dimension $d$ and a general acquisition domain $\Omega$ with cardinality $N_\Omega$
\cite{bronez1988spectral, ha17, siwado11}. Given two parameters, $0<W\leq 1$ (bandwidth) and $K \in \mathbb{N}$ (number of tapers), we first define the \emph{Slepian tapers} recursively, as the set of sequences
$\nu^{(0)} (\Omega,W), \ldots, \nu^{(N_\Omega-1)} (\Omega,W)$ on $\mathbb{Z}^d$ that solve
the following \emph{spectral concentration problem} \cite{sle78}:
\begin{equation*}
\mbox{Maximize} \qquad\int_{[-W/2,W/2]^d}\left|\sum_{n\in\Omega}\nu^{(k)}_{n}(\Omega,W)e^{2\pi i \langle n,\xi\rangle}\right|^2d\xi,
\end{equation*}
\begin{align*}
\text{subject to:}\quad (i)&\quad \sum_{n\in\Omega}|\nu^{(k)}_{n}(\Omega,W)|^2=1,\\ (ii)&\quad \text{supp}\big(\nu^{(k)} (\Omega,W)\big)\subseteq \Omega,\\ (iii)&\quad \nu^{(k)} (\Omega,W)\bot\big\{\nu^{(0)} (\Omega,W),\nu^{(1)} (\Omega,W),...,\nu^{(k-1)} (\Omega,W)\big\}.
\end{align*}
Alternatively, the Slepian tapers are the normalized eigenvectors of the truncated $d$-Toeplitz matrix \cite[Chapters 1 and 3]{hola12}
\begin{align*}
\left\{ \begin{array}{ll} W^d \,
\prod_{k=1}^d \frac{\sin(\pi W(n_k-m_k))}{\pi W(n_k-m_k)}
& \mbox{if}~n, m \in \Omega, \\ 0 & \mbox{otherwise} \end{array} \right..
\end{align*}
The multitaper estimator uses the first $K$ Slepian tapers as masks to build the  averaged periodogram given in \eqref{eq:Smt}.
It is standard to let $K,W$ and $N_\Omega$ be related by
\begin{align}\label{eq_rel}
K = \lceil N_\Omega \cdot W^d \rceil
\end{align}
(smallest integer $\geq N_\Omega \cdot W^d$)
though other choices with $K < \lceil N_\Omega \cdot W^d \rceil$ are of practical interest \cite[Figures 203, 341]{pw93}, \cite[Theorem 5]{krd21}; see also Section \ref{sec_cons}. Corollary \ref{cor:main} below gives an appropriate choice of $K$ as a function of $N_\Omega$, with $W$ being implicitly determined by \eqref{eq_rel}.

\subsection{Spectral-norm mean squared risk rates for multi-tapering}
We formulate risk bounds for multi-tapering in terms of the cardinality of the acquisition domain $N_\Omega$, its \emph{digital perimeter}
\begin{align*}
N_{\partial\Omega}=\sum_{k\in\Z^d}\sum_{j=1}^d |\chi_\Omega(k+e_j)-\chi_\Omega(k)|,
\end{align*}
where $\chi_\Omega$ denotes the characteristic (indicator) function of a domain $\Omega$, and its \emph{diameter} $\mathrm{diam}(\Omega) = \max\{ \|k-j\|_2 : k,j \in \Omega \}$. 
The following result is an analogue of \cite[Theorem 4.3]{anro20} applicable to the error metric \eqref{eq_a} instead of a pointwise error estimate.
\begin{theorem}\label{thm:main}
	Let $X$ be a stationary, real, zero mean, Gaussian process on $\mathbb{Z}^d$ with spectral density $S$ satisfying $\|S\|_{C^2 }\leq 1$.
	Let $\Omega\subset \Z^d$ be a finite acquisition domain with $N_\Omega \geq 3$, and consider the multi-taper estimator \eqref{eq:Smt} with bandwidth $W$ and
	$K = \lceil N_\Omega \cdot W^d \rceil$ tapers. Assume further that
	$N_{\partial\Omega}\geq \left(\frac{N_\Omega }{K}\right)^{1-1/d}$. Then
	\begin{align}\label{eq_mse}
	\E\Big\{\|S - \widehat{S}^{\mathrm{mt}}\|_\infty^2\Big\}
 \lesssim 
	\max_{p\in\{1,2\}} \left(\frac{\log (\emph{diam}(\Omega))}{K}\right)^p + \left(\frac{K}{N_\Omega}\right)^{\frac{4}{d}} +\frac{N_{\partial \Omega}^2 }{N_\Omega^{2-\frac{2}{d}}K^{\frac{2}{d}}}\left[1+\log\left(\frac{N_\Omega}{N_{\partial\Omega}}\right)\right]^2.
	\end{align}
\end{theorem}
We now apply Theorem \ref{thm:main} to a class of acquisition domains
satisfying perimeter bounds similar to those enjoyed by rectangles and disks. This
regime arises, for example, when discretizing an analog domain at increasingly finer scales.
\begin{corollary}\label{cor:main}
	Under the assumptions of Theorem~\ref{thm:main}, 
	suppose that $C>0$ is a constant such that
	\begin{align}\label{eq_cc}
	N_{\partial\Omega} \leq C N_\Omega^{1-1/d},
	\end{align}
	and that  
\begin{equation}\label{eq:cond-diam} 	
 {\emph{diam}(\Omega)\leq \emph{exp}\big( {N_\Omega^{ 1/d}}\big)}.
\end{equation}
	Then, for $d=1$, the choice $K=\left\lceil\big(\log  (\emph{diam}(\Omega))\cdot  N_\Omega^{4}\big)^{1/5} \right\rceil$ yields the upper bound
	\begin{equation}\label{eq_aa}
	\E\left\{\norma{\infty}{S - \widehat{S}^{\mathrm{mt}}}^2\right\}
	\lesssim  \left(\frac{\log(\emph{diam}(\Omega))}{N_\Omega} \right)^{4/5},
	\end{equation}
	while, for $d\geq 2$,  the following bound holds for $K=\left\lceil \big( \log (\emph{diam}(\Omega))^d\cdot N_\Omega^2\big)^{1/3}\right\rceil$:
	\begin{equation}\label{eq_bb}
\E\left\{\norma{\infty}{S - \widehat{S}^{\mathrm{mt}}}^2\right\}\lesssim  \left(\frac{\log (\emph{diam}(\Omega))}{N_\Omega^{1/d}} \right)^{4/3}.
	\end{equation}
	The implied constants in \eqref{eq_aa} and \eqref{eq_bb} depend on the constant $C$ in \eqref{eq_cc}.
\end{corollary}
\begin{rem}[\emph{Risk optimality of Thomson's multi-taper}]
Applying Corollary \ref{cor:main} to a one-dimensional interval $\Omega=\{1, \ldots, N\}$, we see that Thomson's classical multi-taper estimator \cite{th82} achieves the minimax error rate \eqref{eq_3}.
\end{rem}

\begin{rem}
In Theorem \ref{thm:main} and Corollary \ref{cor:main} the number of tapers $K$ and the bandwidth parameter $W$ are linked by \eqref{eq_rel}. In terms of $W$, \eqref{eq_mse} reads
$$
	\E\Big\{\|S - \widehat{S}^{\mathrm{mt}}\|_\infty^2\Big\} \lesssim 
	\max_{p\in\{1,2\}}{\left(\frac{\log (\text{diam}(\Omega))}{N_\Omega W^d}\right)^p}+W^4+\frac{N_{\partial \Omega}^2 }{N_\Omega^{2}W^2}\left[1+\log\left(\frac{N_\Omega}{N_{\partial\Omega}}\right)\right]^2 ,
	$$
while the error rates \eqref{eq_aa} and \eqref{eq_bb} hold when
\[W=\left\{\begin{array}{rl}
N_\Omega^{-1/5}\ \big[\log (\text{diam}(\Omega))\big]^{1/5}, & d=1
\\
N_\Omega^{- {1}/{3d}}\big[\log (\text{diam}(\Omega))\big]^{1/3},& d\geq 2
\end{array}\right..\]
\end{rem}

\subsection{Lower risk bounds for general acquisition domains}
We are unaware of precise minimax rates for the covariance estimation problem with general acquisition domains. As a first benchmark, we derive the following bound. 
\begin{theorem}\label{thm:lower-minimax}
	Let $\Omega\subset\Z^d$ be finite with $N_\Omega \geq 3$. Then
	\begin{equation}
	\inf_{\widehat{S}}\sup_{S}
		\E\left\{\big\|S - \widehat{S} \big\|_\infty^2\right\} \gtrsim \left(\frac{ \log N_\Omega  }{\emph{diam}(\Omega)^d}\right)^{\frac{4}{4+d}},
	\end{equation}
	where the supremum is taken over all sationary, real, zero mean, Gaussian processes on $\mathbb{Z}^d$ with spectral density satisfying $\norm{C^2}{S} \leq 1$, and the infimum runs over all estimators $\widehat{S}$ based on an observation of $X$ on $\Omega$. In particular, if $\mathrm{diam}(\Omega) \leq  C N_\Omega^{1/d}$, then
	\begin{equation}\label{eq_ccc}
	\inf_{\widehat{S}}\sup_{S}\E\left\{\big\|S-\widehat S\big\|_\infty^2\right\}\gtrsim \left(\frac{\log N_\Omega }{N_\Omega}\right)^{\frac{4}{4+d}},
	\end{equation}
	where the implied constant depends on $C$.
\end{theorem}
\begin{rem}[\emph{Almost risk optimality of multi-tapering for certain two dimensional domains}]
For classes of two dimensional acquisition domains $\Omega$ satisfying
$\mathrm{diam}(\Omega) \leq C  N_\Omega^{1/2}$, the upper bound for the multi-taper estimator \eqref{eq_bb} deviates from the minimax risk bound \eqref{eq_ccc} only by a factor $(\log N_\Omega)^{2/3}$.
\end{rem}

\subsection{Proofs}
The upper bounds are obtained by combining the bias bounds from \cite{abro17,anro20} with methods of concentration of measure, as used in \cite{carezh13, kara17}. More specifically,
\cite{abro17} and \cite{anro20} bound the mean squared error $\E\big\{|S(\xi) - \widehat{S}^{\mathrm{mt}}(\xi)|^2\big\}$ corresponding to the estimation of individual spectral frequencies $\xi$. Specifically, under the assumptions of Theorem \ref{thm:main}, \cite[Theorem 4.3]{anro20} gives
\begin{align}\label{eq_xxx}
	\sup_{\xi \in \mathbb{R}^d} \E\Big\{ |S(\xi) - \widehat{S}^{\mathrm{mt}}(\xi) |^2 \Big\}
	\lesssim 
	\frac{1}{K}  + \left(\frac{K}{N_\Omega}\right)^{\frac{4}{d}} +\frac{N_{\partial \Omega}^2 }{N_\Omega^{2-\frac{2}{d}}K^{\frac{2}{d}}}\left[1+\log\left(\frac{N_\Omega}{N_{\partial\Omega}}\right)\right]^2.
\end{align}
 Those estimates are here extended to the stronger error metric $\E\big\{\norma{\infty}{S - \widehat{S}^{\mathrm{mt}}}^2\big\}$, by combining them with concentration estimates for covariance estimators. Comparing \eqref{eq_mse} to \eqref{eq_xxx}, we see that the error bound for the uniform norm has an additional logarithmic factor.
 
 The proofs of the risk bounds for the multi-taper estimator require reinspecting and adapting a portion of the arguments in \cite{carezh13, kara17}. Following this path, it would also be possible to derive deviation bounds as in \cite[Theorem 3]{kara17} (which are stronger than MSE bounds); see also \cite[Theorem 5]{krd21}.
The lower (minimax risk) bounds are obtained by direct adaptation of \cite{carezh13}. Detailed proofs are provided below.

\section{Conclusions}\label{sec_cons}
Corollary \ref{cor:main} shows that Thomson's multi-taper has optimal
risk rates for covariance estimation of stationary Gaussian processes with $C^2$ spectral densities when the bandwidth parameter $W$, the number of samples $N$, and the number of tapers $K$ are linked by $K \asymp N^{4/5} \log(N)^{1/5}$ and $K=\lceil N_\Omega W\rceil$. (Note that in many references the bandwidth interval is $[-W,W]$, so that the previous equation reads $K=\lceil 2 N_\Omega W\rceil$.) However, in certain situations, notably when the spectral density exhibits a high dynamic range, practitioners choose $K< N_\Omega W$. Comments to such effect can be found in \cite{pw93} and are illustrated in \cite[Figures 203, 341]{pw93}. The recent article \cite{krd21} provides analytic and numerical evidence for the superiority of certain choices $K <  N_\Omega W$ in various setups, including non-smooth spectral densities, or moderate values of $N_\Omega$, when the implied constants in \eqref{eq_a} cannot be neglected; see \cite[Theorems 2 and 5]{krd21}. We thus contribute to the discussion on the optimal number of tapers in Thomson's method by showing that the choice $K=\lceil  N_\Omega W\rceil$ meets the benchmark derived in \cite{carezh13}, and hope to motivate the investigation of more adequate formal benchmarks to reflect the possible practical advantages of other choices of $K$.

Multi-taper estimators for general acquisition domains are relevant in many areas of applications including geophysics, see, e.g., \cite{hasi12, simons2003spatiospectral,siwado11}, and are comparatively less explored than Thomson's original estimator. For general acquisition geometries, we have shown that the estimates for single frequencies in \cite{abro17} and \cite{anro20} extend to spectral norm estimates with a logarithmic gain in the corresponding bounds. In dimension $d=1,2$, Theorem \ref{thm:lower-minimax} shows that such bounds are almost optimal for classes of domains whose diameter is suitable dominated. We do not know the precise mini-max rate for covariance estimation in the spectral norm outside those regimes.

\section{Detailed proofs}
The proofs in this section follow closely \cite{carezh13, kara17}, and in some cases provide more details and simplified or optimized arguments.

\subsection{Proof of upper bounds}
Consider  a stationary Gaussian process  $X$ on $\mathbb{Z}^d$.
The multi-taper estimator \eqref{eq:Smt} can be rewritten as
\begin{align}\label{eq_bbb}
\widehat{S}^{\mathrm{mt}}(\xi)=\sum_{\ell\in (\Omega-\Omega)}e^{-2\pi i\langle\xi, \ell\rangle}\left(\sum_{\substack{n,m\in\Omega\\ n-m=\ell}} X_n X_{m}\frac{1}{K}\sum_{k=0}^{K-1}\nu^{(k)}_n(\Omega,W) \nu^{(k)}_{m}(\Omega,W)\right),
\end{align}
where $A-B=\{a-b:\ a\in A,\ b\in B\}$.
This shows that $\widehat{S}^{\mathrm{mt}}$ (and consequently also $\mathbb{E}\{\widehat{S}^{\mathrm{mt}}\}$) is a multivariate trigonometric polynomial of maximum component degree  $\omega:=\lceil\text{diam}(\Omega)\rceil$. 
We will use the following sampling inequality: for a trigonometric polynomial $p(\xi) = \sum_{l \in \{-n,\ldots,n\}^d} c_l e^{-2\pi i\langle\xi, \ell\rangle}$, we have $\sup_{\xi \in [0,1]^d} |p(\xi)| \leq C_d \max_{ \ell\in\{0,...,4n-1\}^d} |p(\tfrac{\ell}{4n})|$, where $C_d$ is a constant that depends on $d$. 
The one dimensional version of the sampling inequality is classical (see, e.g., \cite[Chapter 5.2]{MR0094840}) while the general case follows by applying the one dimensional result to each variable iteratively; alternatively, the multi-dimensional result can deduced from \cite[Theorem~1]{pfbr18}.

For $\ell\in \{1,\ldots,4\omega\}^d$ we set $(\xi_\ell)_k:=\frac{\ell_k-1}{4\omega}$. Applied to \eqref{eq_bbb}, the sampling inequality yields:
\begin{align}\label{eq:trig-est}
\|\widehat{S}^{\mathrm{mt}}-\mathbb{E}\{\widehat{S}^{\mathrm{mt}}\}\|_\infty&=\sup_{\xi\in [0,1]^d}|\widehat{S}^{\mathrm{mt}}(\xi)-\mathbb{E}\{\widehat{S}^{\mathrm{mt}}\}(\xi)|  \nonumber
\\
&\leq C_d \max_{ \ell\in\{1,...,4\omega\}^d}\left|\widehat{S}^{\mathrm{mt}}\left(\xi_\ell\right)-\mathbb{E}\{\widehat{S}^{\mathrm{mt}}\}\left(\xi_\ell\right)\right|,
\end{align}
where the first identity follows from the $\mathbb{Z}^d$-periodicity of $\widehat{S}^{\mathrm{mt}}$.

We now express the multi-taper as done in \cite{liro08}.
In terms of the matrix-valued function $V_K:[0,1]^d\rightarrow \C^{\Omega\times \Omega}$,
\begin{align}\label{eq_V}
\big(V_K(\xi)\big)_{n,m}:=e^{-2\pi i\langle \xi, n-m\rangle}\frac{1}{K}\sum_{k=0}^{K-1}\nu^{(k)}_n(\Omega,W)\cdot \nu^{(k)}_m(\Omega,W), \qquad n,m \in \Omega,
\end{align}
the multi-taper estimator can be written as a quadratic operation on the restricted process $X_{|\Omega}$:
\begin{align*}
\widehat{S}^{\mathrm{mt}}(\xi)=\langle V_K(\xi)X_{|\Omega},X_{|\Omega}\rangle.
\end{align*}
(We note that this formula does not correspond to ``tapering'' as considered in \cite{kara17}.)
At grid points, the deviation of the multi-taper estimator from its mean is therefore
\begin{align}\label{eq_dd}
Z_K^\ell:=\widehat{S}^{\mathrm{mt}}\left(\xi_\ell\right)-\mathbb{E}\{\widehat{S}^{\mathrm{mt}}\}\left(\xi_\ell\right)=\langle V_K(\xi_\ell){X_{|\Omega}},{X_{|\Omega}}\rangle-\mathbb{E}\langle V_K(\xi_\ell){X_{|\Omega}},{X_{|\Omega}}\rangle.
\end{align}
As in \cite{kara17}, the main tool to analyze \eqref{eq_dd} is the following generalization of the Hanson-Wright inequality, see \cite{abra15,ruve13}.
\begin{theorem}\label{thm:hansonwright}
	Let $Y\sim \mathcal{N}(0,\Sigma)$, then for every $A\in\R^{N\times N}$ and every $t>0$:
\begin{equation}\label{eq:hw}
	\P\big(|\langle AY,Y\rangle -\E\{\langle AY,Y\rangle\} |\geq t\big)\leq 2\ \emph{exp}\left[-\frac{1}{C}\min\left(\frac{t^2}{\|\Sigma\|_s^2\|A\|_F^2},\frac{t}{\|\Sigma\|_s\|A\|_s}\right)\right]
\end{equation}
	for some universal constant $C>0$.
\end{theorem}
(Here, $\|A\|_F$ is the Hilbert-Schmidt (Frobenius) norm of $A$, while $\|A\|_s = \sup_{\|x\|_2 \leq 1} \|Ax\|_2$ is the spectral norm.)
To apply \eqref{eq_dd} we first carry out the following calculation.
\begin{lemma}\label{lem:HS}
	The Frobenius norm of   $V_K(\xi_\ell)$ equals $1/\sqrt{K}$ and its spectral norm is bounded by $1/K$.
\end{lemma}
\begin{proof}
Direct calculations yield
\begin{align*}
\|V_K(\xi_\ell)\|_{F}^2&=\sum_{n \in \Omega}\sum_{m \in \Omega}\left|e^{-2\pi i\langle\xi_\ell,n-m\rangle}\frac{1}{K}\sum_{k=0}^{K-1}\nu^{(k)}_{n}(\Omega,W) \nu^{(k)}_{m}(\Omega,W)\right|^2\\
&=\frac{1}{K^2}\sum_{n\in\Omega} \sum_{m\in\Omega} \sum_{k=0}^{K-1}\sum_{s=0}^{K-1}\nu^{(k)}_n(\Omega,W) \nu^{(k)}_m(\Omega,W)\nu^{(s)}_n(\Omega,W) \nu^{(s)}_m(\Omega,W)\\
&=\frac{1}{K^2}\sum_{k=0}^{K-1}\sum_{s=0}^{K-1}\left(\sum_{n\in \Omega} \nu^{(k)}_n(\Omega,W)\nu^{(s)}_n(\Omega,W)\right)\left(\sum_{m\in\Omega}\nu^{(k)}_m(\Omega,W)\nu^{(s)}_m(\Omega,W)\right)\\
&=\frac{1}{K^2}\sum_{k=0}^{K-1}\sum_{s=0}^{K-1}\delta_{s,k}=\frac{1}{K}.
\end{align*}
To estimate the spectral norm, we first note that $V_K(\xi_\ell)$ is a positive semi-definite matrix. Therefore, using the notation $(M_\xi a)_n=e^{2\pi i \langle\xi, n\rangle }a_n,\ n\in\Omega,$
\begin{align*}
\|V_K(\xi_\ell)\|_s&=\sup_{\|Y\|_2=1}\langle V_K(\xi_\ell)Y,Y\rangle=\sup_{\|Y\|_2=1}\frac{1}{K}\sum_{k=0}^{K-1}|\langle Y,M_{\xi_\ell} \nu^{(k)}(\Omega,W)\rangle |^2
\\
&\leq \sup_{\|Y\|_2=1}\frac{1}{K}\sum_{k=0}^{N_\Omega-1}|\langle M_{-\xi_\ell}  Y,\nu^{(k)}(\Omega,W)\rangle |^2
=\sup_{\|Y\|_2=1}\frac{1}{K} \| M_{-\xi_\ell}  Y\|_2^2=\frac{1}{K},
\end{align*}
as the Slepian tapers form an orthonormal basis.
\end{proof}
We can now prove the main result.
\begin{proof}[Proof of Theorem \ref{thm:main}]\mbox{}
	
\noindent {\bf Step 1}. \emph{Bias estimates}. The following estimate was shown in  \cite[Proof of Theorem 4.3]{anro20},
\begin{align}\label{eq_bias}
\begin{aligned}
|\text{Bias}(\widehat{S}^{\mathrm{mt}})(\xi)|&=|\E\{\widehat{S}^{\mathrm{mt}}\}(\xi)-S(\xi)|\\ &\lesssim  \|S\|_{C^2 }\left(\left(\frac{K}{N_\Omega}\right)^{2/d}+\frac{N_{\partial \Omega} }{N_\Omega^{1-1/d}K^{1/d}}\left[1+\log\left(\frac{N_\Omega}{N_{\partial\Omega}}\right)\right]\right).
\end{aligned}
\end{align}
(Here it is essential that $K = N_\Omega W^d + O(1)$.)

\noindent {\bf Step 2}. \emph{Concentration of the estimator}. We show that
\begin{align}\label{eq_var}
\E\big\{\|\widehat{S}^{\mathrm{mt}}- \E\{\widehat{S}^{\mathrm{mt}}\big\}\|_\infty^2\}\lesssim  \|S\|^2_\infty\cdot\left( \frac{\log(\text{diam}(\Omega))}{K}+\left(\frac{ \log(\text{diam}(\Omega))}{K}\right)^2\right) .
\end{align}
Recall that $\omega=\lceil\text{diam}(\Omega)\rceil$. 
As $N_\Omega \geq 3$,
\begin{align}\label{eq_d}
\omega \asymp \text{diam}(\Omega) \geq c_d>1,
\end{align}
for a dimensional constant $c_d$. 

For $\ell\in \{1,\ldots,4\omega\}^d$, consider the matrices $V_K(\xi_\ell)$ defined in \eqref{eq_V} and the random vectors
$Z_K^\ell$, defined in \eqref{eq_dd}. Let $\Sigma_{|\Omega}$ be the covariance of $X_{|\Omega}$. Then
\begin{align}\label{eq_S}
\|\Sigma_{|\Omega}\|_s \leq \| \Sigma \|_s= \|S\|_\infty.
\end{align}
We apply Theorem~\ref{thm:hansonwright} with 
$Y=X_{|\Omega}$ and $A=V_K(\xi_\ell)$, to obtain a tail bound on each $Z_K^\ell$ and use Lemma~\ref{lem:HS} to conclude
\begin{align}\label{eq:hw-applied} 
\P(|Z_K^\ell|\geq t)&\leq   2\cdot \text{exp}\left[-\frac{1}{C}\min\left(\frac{K t^2}{\|\Sigma\|^2_s },\frac{K t}{\|\Sigma\|_s }\right)\right].
\end{align}
Using \eqref{eq:hw-applied}, the trivial bound $\P(A)\leq 1 $, and the union bound  leads to
\begin{align*}\label{eq:tail-sg1}
\P\left(\max_{\ell\in\{1,...,4\omega\}^d} |Z_K^\ell|\geq t\right)&\lesssim  \min\left(1,  \omega^d\cdot \text{exp}\left[-\frac{1}{C}\min\left(\frac{K t^2}{\|\Sigma\|^2_s },\frac{K t}{\|\Sigma\|_s }\right)\right]\right).
\end{align*}
We now invoke the   tail bound above together with the sampling bound \eqref{eq:trig-est} to estimate the variance term. In doing so, we have to distinguish the cases (i) $Cd\log\omega\leq K$ and (ii) $Cd\log\omega> K$. In the former case,
\allowdisplaybreaks[3]
\begin{align*}
\E\big\{\| \widehat{S}^{\mathrm{mt}}- \E\{\widehat{S}^{\mathrm{mt}}\big\}\|_\infty^2\}&\lesssim \ \E\left\{\max_{ \ell\in\{1,...,4\omega\}^d}|\widehat{S}^{\mathrm{mt}}(\xi_\ell)- \E\{\widehat{S}^{\mathrm{mt}}\}(\xi_\ell)|^2\right\}\nonumber 
\\
&=  \E\left\{\max_{\ell\in\{1,...,4\omega\}^d}|Z_K^\ell|^2\right\}\nonumber \\ &=2\int_0^\infty t\cdot \P\left(\max_{\ell\in\{1,...,4\omega\}^d}|Z_K^\ell|>t\right)dt \nonumber 
\\
&\lesssim \int_0^{\|\Sigma\|_s\sqrt{\frac{ C{d\log\omega}}{ {K}}}}t\ dt+\omega^d\int_{\|\Sigma\|_s\sqrt{\frac{ C{d\log\omega}}{ {K}}}}^{\|\Sigma\|_s}  t\cdot \text{exp}\left[-\frac{1}{C} \frac{{K} t^2}{\|\Sigma\|^2_s }\right]dt
\\
&\qquad\qquad\qquad\qquad\qquad\quad +\omega^d\int_{\|\Sigma\|_s}^\infty  t\cdot \text{exp}\left[-\frac{1}{C} \frac{{K} t}{\|\Sigma\|_s }\right]dt
\\
&\leq\frac{C  \|\Sigma\|^2_s d \log\omega}{2K}+\frac{C\|\Sigma\|^2_s\omega^d}{2K}\int_{d\log\omega}^\infty   e^{-t}dt+ \frac{(C\|\Sigma\|_s)^2\omega^d}{K^2} \int_{\frac{K}{C}}^\infty t\cdot e^{-t}dt
\\
&=\frac{C  \|\Sigma\|_s^2 d \log\omega}{2K}+\frac{C\|\Sigma\|^2_s}{2K}+\frac{(C\|\Sigma\|_s)^2\omega^d}{K^2}\cdot\Gamma\Big(2,\frac{K}{C}\Big),
\end{align*}
where $\Gamma(s,x)$ denotes the upper incomplete gamma function. As $\Gamma(s+1,x)=s\Gamma(s,x)+x^se^{-x}$, we find
\begin{equation}\label{eq:log_l_K}
\E\big\{\| \widehat{S}^{\mathrm{mt}}- \E\{\widehat{S}^{\mathrm{mt}}\big\}\|_\infty^2\}\lesssim
\frac{ \|\Sigma\|^2_s}{K}\left(\log\omega +1+ {  \omega^d} e^{-K/C}\right)\lesssim\frac{ \|\Sigma\|^2_s\log\omega}{K},
\end{equation}
where the last inequality follows from the assumption $Cd\log\omega\leq K$.

If $Cd\log \omega>K$, then, following the same line of arguments, we obtain
\begin{align} \label{eq:log_g_K}
\E\big\{\| \widehat{S}^{\mathrm{mt}}- \E\{\widehat{S}^{\mathrm{mt}}\big\}\|_\infty^2\}&
\lesssim \int_0^{ \frac{ C\|\Sigma\|_s{d\log\omega}}{ {K}}}t\ dt+\omega^d\int_{ \frac{ C\|\Sigma\|_s{d\log\omega}}{ {K}}}^{\infty}  t\cdot \text{exp}\left[-\frac{1}{C} \frac{{K} t}{\|\Sigma\|_s }\right]dt \nonumber
\\
&=\frac{1}{2}\left(\frac{C  \|\Sigma\|_s\log\omega}{K}\right)^2 + \frac{(C\|\Sigma\|_s)^2\omega^d}{K^2} \int_{d\log\omega}^\infty t\cdot e^{-t}dt
\nonumber
\\
&=\frac{1}{2}\left(\frac{C  \|\Sigma\|_s\log\omega}{K}\right)^2+\frac{(C\|\Sigma\|_s)^2\omega^d}{K^2}\cdot\Gamma\Big(2,d\log\omega\Big)
\nonumber 
\\
&\lesssim \left(\frac{\|\Sigma\|_s\log\omega}{K}\right)^2.
\end{align}
The estimates \eqref{eq:log_l_K} and \eqref{eq:log_g_K}, together with \eqref{eq_S}, yield \eqref{eq_var}.

\noindent {\bf Step 3}. We decompose the mean squared error as
\begin{align*}
\E\left\{\norma{\infty}{S - \widehat{S}^{\mathrm{mt}}}^2\right\}&\lesssim \E\big\{\|\widehat{S}^{\mathrm{mt}}-\E\{\widehat{S}^{\mathrm{mt}}\}\|_\infty^2\big\}+\|\E\{ \widehat{S}^{\mathrm{mt}}\}-S\|_\infty^2
\end{align*}
and combine \eqref{eq_bias} and \eqref{eq_var} to obtain \eqref{eq_mse}. 
\end{proof}

\begin{proof}[Proof of Corollary \ref{cor:main}]\mbox{}
For every set $\Omega\subset \Z^d$ one has $ N_\Omega\leq (2\cdot\text{diam}(\Omega)+1)^d\lesssim \text{diam}(\Omega)^d$, where the second inequality uses that $\text{diam}(\Omega)>1$ (as $N_\Omega\geq 3$).  It then follows from \eqref{eq_mse} and \eqref{eq_cc} that 
$$
\E\left\{\norma{\infty}{S - \widehat{S}^{\mathrm{mt}}}^2\right\} \lesssim  \frac{\log (\text{diam}(\Omega))}{K}+\frac{(\log (\text{diam}(\Omega)))^2}{K^{\frac{2}{d}}}+\left(\frac{K}{N_\Omega}\right)^{\frac{4}{d}}.
$$
If $d=1$, we assume for the moment that $\log(\text{diam}(\Omega))\leq K$ which implies
$$
\E\left\{\norma{\infty}{S - \widehat{S}^{\mathrm{mt}}}^2\right\} \lesssim   \frac{\log (\text{diam}(\Omega))}{K} +\left(\frac{K}{N_\Omega}\right)^{4}.
$$
Optimizing the right hand side with respect to $K$ yields the possibly non-integer value $K^\ast=\left(\frac{1}{4}\log(\text{diam}(\Omega)) N_\Omega^{4}\right)^{1/5}$. If we  choose $K$ slightly larger than the optimum, namely $K=\left\lceil \left(\log(\text{diam}(\Omega)) N_\Omega^{4}\right)^{1/5}\right\rceil$, then 
$$
\frac{\log(\text{diam}(\Omega))}{K}\leq\left(\frac{\log(\text{diam}(\Omega))}{N_\Omega}\right)^{4/5}\leq 1,
$$
 by the assumption \eqref{eq:cond-diam}. A direct computation yields \eqref{eq_aa}.

Note that there are at most $N_\Omega$ orthogonal tapers. It is therefore necessary to ensure that our choice of $K$ satisfies $K\leq N_\Omega$ which is also done by \eqref{eq:cond-diam}.

For $d\geq 2$, one has $\log(\text{diam}(\Omega))/K\leq \log(\text{diam}(\Omega))^2/K^{2/d}$ showing that 
$$
\E\left\{\norma{\infty}{S - \widehat{S}^{\mathrm{mt}}}^2\right\} \lesssim   \frac{(\log \text{diam}(\Omega))^2}{K^{\frac{2}{d}}}+\left(\frac{K}{N_\Omega}\right)^{\frac{4}{d}},
$$
and the previous optimization steps can be repeated to finish the proof.
\end{proof}

\subsection{Proof of lower bounds} 
The \emph{Kullback-Leibler divergence} of two 
probability distributions $\mathbb{P},\mathbb{Q}$ is defined by
$$
K(\mathbb{P},\mathbb{Q})=\int \log\frac{\text{d}\mathbb{P}}{\text{d}\mathbb{Q}}\text{d}\mathbb{P},
$$
provided that $\mathbb{P}$ is absolutely continuous with respect to $\mathbb{Q}$
(denoted $\mathbb{P} \ll \mathbb{Q}$). For $d$-dimensional normal distributions, the Kullback-Leibler divergence reads
\begin{equation}\label{eq:kull-norm}
K(\mathcal{N}(0,\Sigma_1),\mathcal{N}(0,\Sigma_2))=\frac{1}{2}\left(\text{trace}(\Sigma_2^{-1}\Sigma_2)-d-\log\left(\frac{\text{det}\Sigma_1}{\text{det}\Sigma_2}\right)\right).
\end{equation}
We will need the following version of Fano's lemma which can be found for instance in \cite[Theorem 2.7]{tsy09}.
\begin{prop}\label{thm:fano}
	Let $\mathcal{S}_M=\{S_0,S_1,...,S_M\}$ be a family of bounded functions on $\R^d$, $M\geq 1$, $\{Z_0,...,Z_M\}$ a set of random variables with corresponding probability measures $\{\mathbb{P}_0,...,\mathbb{P}_M\}$, and $d:\mathcal{S}_M\times \mathcal{S}_M\rightarrow [0,\infty)$ a metric. If  
	\begin{enumerate}
		\item[(i)] $d(S_i,S_j)\geq 2\delta>0$,\quad $0\leq i<j\leq M$,
		\item[(ii)] $\mathbb{P}_j\ll \mathbb{P}_0$,\quad $j=1,...,M$, and for some  $0<\alpha<1/8$
		\begin{equation}\label{eq:tsyba-kull-sum}
		\sum_{j=1}^M K(\mathbb{P}_j,\mathbb{P}_0)\leq \alpha M \log M,
		\end{equation}
	\end{enumerate}
	then we have
	\begin{equation}\label{eq:tsyba-concl}
	\inf_{\widehat{S}}\sup_{S\in\mathcal{S}_M}\E\left\{d(\widehat{S},S)^2\right\}\geq C \delta^2,
	\end{equation}
	where the infimum is taken over all estimators based on an observation of $Z_0,...,Z_M$ and
	$C$ only depends on $\alpha$.
\end{prop}

Let $\mathcal{S}$ be the class of all $C^2$ functions on the line, with $\|S\|_{C^2}\leq 1$, that are the spectral densities of zero mean, stationary, stochastic Gaussian processes on $\Z^d$. Each function in $\mathcal{S}$ is in fact the spectral density of a \emph{unique} stationary, real, zero mean, Gaussian process (which is necessarily ergodic).

 The partial Fourier sum of a function $S$ is denoted by 
$$
F_p(S)(\xi)=\sum_{\|k\|_\infty\leq p} \mathcal{F}(S)(k) e^{2\pi i \langle \xi,k\rangle}.
$$
The following lemma provides a class of functions that allows us to apply Fano's lemma.

\begin{lemma}
\label{lemma_fano_class}
There exists $\varepsilon>0$ such that for all $M \in \mathbb{N}$ and $\tau \in (0,\varepsilon)$ a family of functions $
\mathcal{S}_M=\{S_0,S_1,...,S_M\} \subseteq \mathcal{S}$ with the following properties exists:
\begin{enumerate}
	\item[(i)] $S_0 \equiv 1/2$.
	\item[(ii)] $\|S_n-S_0\|_\infty \asymp \|S_n-S_m\|_\infty \asymp   \frac{\tau}{M^{2/d}}, \quad 1 \leq n\not=m \leq M$,
	\item[(iii)]  $\|S_n-S_0\|_2^2 \asymp    \frac{\tau^2}{M^{1+4/d}}, \quad 1 \leq n \leq M$,
	\item[(iv)] $\|F_p(S_n)-S_0\|_\infty \leq 1/4, \quad 1 \leq n \leq M, \quad p \geq 0,$
	\item[(v)]  $F_p(S_n)\geq 0,\quad 1 \leq n \leq M, \quad p \geq 0.$
\end{enumerate}
(The implicit constants and $\varepsilon$ only depend on the ambient dimension $d$.)
\end{lemma}
\begin{proof}
\noindent {\bf Step 1}.
Let us define  $A\in C^\infty(\R^d)$  by
$$
A(x)=\text{exp}\left(-\frac{1}{1-4\|x\|^2_2}\right)\chi_{\{\|x\|_2< 1/2\}}(x),\quad x\in\R^d.
$$
Set $K:=\lceil M^{1/d}\rceil+1$  and, for $n\in\{1,...,K-1\}^d$,
\begin{align}\label{eq_b}
S_n(\xi):=S_0+\frac{\tau}{K^2} \sum_{e\in\{-1,1\}^d}  A\big(2K\xi-\varphi(n,e)\big), \quad \xi\in [-1/2,1/2]^d,
\end{align}
where $\varphi:\R^d\times\R^d\rightarrow \R^d$ is defined by
$$
\varphi(x,y)_j=\left\{\begin{array}{ll}\ \  x_j, & \text{if}\ y_j\geq 0 \\ -x_j, & \text{if}\ y_j<0 \end{array}\right. .
$$
Now  pick $M$ functions out of $\{S_n\}_{n\in\{1,...,K-1\}^d}$, define (after relabelling the indices) 
$\mathcal{S}_M:=\{S_0,S_1,...,S_M\}$, choose $\varepsilon$ such that
$
\|S_n\|_{C^2 }\leq 1/2+4\varepsilon\|A\|_{C^2 }\leq 1,
$
and let $\tau \in (0,\varepsilon)$.

\smallskip

\noindent {\bf Step 2}.
Let us show that the Fourier coefficients of $S_n$ only depend on the componentwise magnitude of their index:
\begin{align}\label{eq_aaa}
\F(S_n)(m)=\F(S_n)\left((|m_1|, \ldots, |m_d|)\right), \qquad m \in \mathbb{Z}^d.
\end{align}
We first note that $\varphi$ is linear in the first argument and that $\varphi(\varphi(x,y),y)=x$. Moreover, if $y_j\neq 0,$ for every $j=1,...,d$, then $\varphi(\varphi(x,y),z)=\varphi(x,\varphi(y,z))$. These properties together with symmetry of $A$ w.r.t. the origin imply that 
\begin{align*}
A\big(2K\varphi(\xi,m) -\varphi(n,e)\big)&=A\Big(2K \varphi(\xi,m) -\varphi\big(\varphi(\varphi(n,e),m),m\big) \Big)
\\&=A\Big(\varphi\big(2K\xi -\varphi(\varphi(n,e),m),m\big)\Big)
\\&=A\big( 2K\xi -\varphi(\varphi(n,e),m) \big) 
\\ &=A\big(2K\xi -\varphi(n,\widetilde e)\big),
\end{align*}
for $\widetilde e:=\varphi(e,m)\in \{-1,1\}^d$. Since $\varphi(\cdot,y):\{-1,1\}^d\rightarrow \{-1,1\}^d$ is bijective for every $y\in\R^d$, we conclude that
$S_n(\varphi(\xi,m))=S_n(\xi)$. Hence, writing $|m|=(|m_1|, \ldots, |m_d|)$,
\begin{align*}
\F(S_n)(m)&=\int_{\big[-\frac{1}{2},\frac{1}{2}\big]^d}S_n(\xi)e^{2\pi i\langle \xi,m\rangle} d\xi=\int_{\big[-\frac{1}{2},\frac{1}{2}\big]^d}S_n(\xi)e^{2\pi i\langle \varphi(\xi,m),|m|\rangle} d\xi
\\
&=\int_{\big[-\frac{1}{2},\frac{1}{2}\big]^d}S_n(\varphi(\xi,m))e^{2\pi i\langle \xi,|m|\rangle} d\xi=\F(S_n)(|m|).
\end{align*}

\smallskip

\noindent {\bf Step 3}. Note that
for $e\in\{-1,1\}^d$ and $n\in\{1,\ldots,K-1\}^d$,
$$
\text{supp}\big[ A\big(2K\xi-\varphi(n,e)\big) \big] \subseteq B_{1/4K}\big( \varphi(n,e)/2K\big).
$$
As  $B_{1/4K}\big( \varphi(n_1,e_1)/2K\big)$ and  $B_{1/4K}\big( \varphi(n_2,e_2)/2K\big)$
are disjoint whenever   $e_1\neq e_2$ or $n_1\neq n_2$, a direct inspection of \eqref{eq_b} yields (ii) and (iii).  

In addition, $\|S_0-S_n \|_{C^2} \leq C\tau \leq C \varepsilon$, and therefore
$\|F_p(S_0-S_n)\|_\infty = \|F_p(S_n)-S_0\|_\infty \leq C \varepsilon$, for a constant $C>0$ --- see, e.g., \cite[Theorem 4.4]{schu69}. Hence, (iv) holds, provided that $\varepsilon$ is chosen small enough. Finally, by \eqref{eq_aaa}, $F_p(S_n)\in\R$ for every $p \geq 0$,
so (iv) implies (v).
\end{proof}
We can now prove the desired lower bounds.

\begin{proof}[Proof of Theorem \ref{thm:lower-minimax}]
Set $\omega:=\lceil\text{diam}(\Omega)\rceil$ and
\begin{align}\label{eq_M}
M:=\left\lceil \left(\frac{\omega}{(\log\omega)^{1/d}}\right)^{d^2/(4+d)}\right\rceil\geq 2.
\end{align}
Since $N_\Omega\geq 3$ we have $\omega \geq 2$ and consequently
\begin{align}\label{eq_M2}
M \asymp \left(\frac{\omega}{(\log\omega)^{1/d}}\right)^{d^2/(4+d)}.
\end{align}
Fix a small parameter $\tau>0$ to be specified, and invoke Lemma \ref{lemma_fano_class} to obtain functions
$\{S_0,S_1, \ldots, S_M\}$.

\smallskip

\noindent {\bf Step 1}. \emph{More informative model with circulant covariance}. Let $n^\ast\in\Z^d$ be such that 
$$
\Omega\subseteq n^\ast+\{0,...,\omega\}^d=:\widetilde{\Omega}.
$$ 
The samples  $\{X_n: n\in \widetilde\Omega\}$ are distributed according to the covariance matrix $\widetilde{\Sigma}=\Sigma_{|\widetilde\Omega}$. Let $Y$ be the $(2\omega+1)^d$-dimensional Gaussian random variable defined by the multidimensional circulant covariance matrix $\Sigma_Y$:
$$
(\Sigma_Y)_{n,m}:= \sigma_{u(n-m)},\quad m,n\in \{-\omega,\ldots,\omega\}^d,
$$
where   
$$
u(n-m)_j=\left\{\begin{array}{ll}|n_j-m_j|, & \text{ when }|n_j-m_j|\leq \omega,\\
2\omega+1-|n_j-m_j|, & \text{ when }\omega+1\leq |n_j-m_j|\leq 2\omega,\end{array}\right. \quad j=1,...,d.
$$
The minimax error of estimating $S$ from samples of $X$ on $\Omega$ is larger than that based on samples on $\widetilde{\Omega}$. This latter error is in turn larger than the one corresponding to the estimation of $S$ based on samples of $Y$, as the random vector $Y_{|\{0,...,\omega\}^d}$ is  distributed exactly as $X_{|\widetilde{\Omega}}$. Thus,
\begin{align}\label{eq_x1}
\inf_{\widehat{S}_{X_{|\Omega}}}\sup_{S \in \mathcal{S}}
\E\left\{\norma{\infty}{S - \widehat{S}_{X|\Omega} }^2\right\}
\geq \inf_{\widehat{S}_Y}\sup_{S \in \mathcal{S}}
\E\left\{\norma{\infty}{S - \widehat{S}_Y }^2\right\},
\end{align}
where $\widehat{S}_{X_{|\Omega}}$ denotes an estimator based on samples on $\Omega$ of a random process $X$ with spectral density $S$ and $\widehat{S}_Y$ denotes an estimator based on samples of $Y$.

\smallskip

\noindent {\bf Step 2}. \emph{Diagonalization}.
Let $S \in \mathcal{S}$ and assume that $F_\omega(S) \geq 0$. As $\Sigma_Y$ is a multidimensional circulant matrix, it is diagonalized by the multidimensional Fourier matrix \[(U)_{n,m}:=(2\omega+1)^{-d/2}e^{2\pi i \frac{\langle n,m\rangle }{2\omega+1}}, \qquad n,m\in\{-\omega,...,\omega\}^d.\]
Explicitly,
\begin{align*}
U \Sigma_Y U^\ast=\text{diag}\left( \left[
F_\omega(S)(\xi_k) \right]_{k \in \{-\omega, \ldots, \omega\}^d}
\right),
\end{align*}
where
\begin{align}\label{eq_xk}
\xi_k:=\frac{k}{2\omega+1}.
\end{align}
Let $Z:= UY$. Since the covariance matrix $\Sigma_Z=U\Sigma_Y U^\ast$ is diagonal, 
$$
Z_k\sim  [F_\omega(S)(\xi_k)]^{1/2}  \eta_k, \quad \|k\|_\infty \leq \omega,
$$ 
where $\eta_k \stackrel{iid}{\sim} \mathcal{N}(0,1)$.
As the experiment of taking a sample from $Y$ contains exactly as much information as taking a sample from $Z$, we get
\begin{align}\label{eq_x2}
\inf_{\widehat{S}_Y}\sup_{S \in \mathcal{S}}
\E\left\{\norma{\infty}{S - \widehat{S}_Y }^2\right\}
= 
\inf_{\widehat{S}_Z}\sup_{S \in \mathcal{S}}
\E\left\{\norma{\infty}{S - \widehat{S}_Z }^2\right\},
\end{align}
where $\widehat{S}_Z$ denotes an estimator based on samples of $Z$. 

\smallskip

\noindent {\bf Step 3}. \emph{Estimation of Kullback-Leibler divergences}.
Recall that $S_0 \equiv 1/2$ while $F_\omega(S_n)(\xi_k)\geq 0$,
for $n= 1,...,M$. We can therefore define the random vectors $Z_n$ by
$$
(Z_{n})_k=\left[F_\omega(S_n)\left(\xi_k\right)\right]^{1/2}\eta_{k}.
$$
Then, with the notation of \eqref{eq_xk},
\begin{displaymath}
{Z}_n\sim \mathcal{N}\Big(0,\text{diag}\big\{F_\omega(S_{n})\big(\xi_{k}\big)\big\}_{k \in \{-\omega,\ldots,\omega\}^d}\Big).
\end{displaymath}
Let $\mathbb{P}_n$ denote the distribution of $  Z_n$.
By \eqref{eq:kull-norm},
\begin{align*}
K&(\mathbb{P}_n,\mathbb{P}_0)\\ &=\frac{1}{2}\hspace{-0.04cm}\left[\text{trace}\hspace{-0.06cm}\left(\hspace{-0.06cm}\text{diag}\hspace{-0.03cm}\left\{\hspace{-0.04cm}\frac{F_\omega(
	S_{n})\left(\xi_{k}\right)}{S_0}\hspace{-0.04cm}\right\}\hspace{-0.06cm}\right)\hspace{-0.06cm}-\hspace{-0.06cm}(2\omega+1)^d-\log\hspace{-0.03cm}\left(\hspace{-0.03cm}\frac{\text{det}\left(\text{diag}\left\{F_\omega(S_{n})\left(\xi_{k}\right)\right\}\right)}{\text{det}\big(\text{diag}(S_0)\big)}\hspace{-0.03cm}\right)\hspace{-0.04cm}\right]
\\
&=\frac{1}{2}\left[\sum_{\|k\|_\infty\leq  \omega}\frac{F_\omega(S_{n})\left(\xi_k\right)}{S_0}-(2\omega+1)^d-\log\left(\prod_{\|k\|_\infty\leq \omega}\frac{F_\omega(S_{n})\left(\xi_{k}\right)}{S_0}\right)\right]
\\
&=\frac{1}{2}\sum_{\|k\|_\infty\leq  \omega}\left[\frac{F_\omega(S_{n})\left(\xi_k\right)}{S_0}-1-\log\left(\frac{F_\omega(S_{n})\left(\xi_k\right)}{S_0}\right)\right].
\end{align*}
For $\tau$ small enough, we have that $|F_\omega(S_{n})\left(\xi_k\right)/S_0-1|\leq 1/2$. 
Since $a-\log(1+a)\leq a^2$, for $|a|\leq 1/2$, it thus follows  that
$$
\frac{F_\omega(S_{n})\left(\xi_k\right)}{S_0}-1-\log\left(\frac{F_\omega(S_{n})\left(\xi_k\right)}{S_0}\right)\leq \frac{1}{S_0^2}\left(F_\omega(S_{n})\left(\xi_k\right)-S_0 \right)^2.
$$
Let $\nu_{n}(m)$ denote the $m$-th Fourier coefficients of $S_n-S_0$. Then, using  Parseval's identity and $S_0\equiv 1/2$, we estimate
\begin{align*}
\sum_{n=1}^M K(\mathbb{P}_n,\mathbb{P}_0)
&=\frac{1}{2}\sum_{n=1}^M\sum_{\|k\|_\infty\leq  \omega}\left[\frac{F_\omega(S_{n})\left(\xi_k\right)}{S_0}-1-\log\left(\frac{F_\omega(S_{n})\left(\xi_k\right)}{S_0}\right)\right]
\\
&\leq \frac{1}{2S_0^2}\sum_{n=1}^M\sum_{\|k\|_\infty\leq  \omega} |F_\omega(S_{n})\left(\xi_k\right)-S_0 |^2
\\
&=2\sum_{n=1}^M\sum_{\|k\|_\infty\leq  \omega}\left|\sum_{\|m\|_\infty\leq \omega}\nu_{n}(m)e^{2\pi i \langle \xi_k,m\rangle}\right|^2
\\
&=2\sum_{n=1}^M\sum_{\|k\|_\infty\leq  \omega} \sum_{\|m\|_\infty\leq \omega} \sum_{\|\ell\|_\infty\leq  \omega}\nu_{n}(m)\overline{\nu_{n}(\ell)}e^{2\pi i \langle \xi_k,m-\ell\rangle}
\\
&=2\sum_{n=1}^M\sum_{\|m\|_\infty\leq  \omega} \sum_{\|\ell\|_\infty\leq  \omega}\nu_{n}(m)\overline{\nu_{n}(\ell)}\underbrace{\sum_{\|k\|_\infty\leq \omega} e^{2\pi i \langle \xi_k,m-\ell\rangle}}_{=(2\omega+1)^d \delta_{m,\ell}}
\\ 
&=2(2\omega+1)^d \sum_{n=1}^M\sum_{\|m\|_\infty\leq  \omega} |\nu_{n}(m)|^2
 \leq 2 (2\omega+1)^d \sum_{n=1}^M \sum_{m\in\Z^d} |\nu_{n}(m)|^2
\\ 
&=2(2\omega+1)^d \sum_{n=1}^M\|S_n -S_0\|^2_2 .
\end{align*}
Invoking Lemma~\ref{lemma_fano_class}~(ii) we conclude
\begin{align}\label{eq_c}
\sum_{n=1}^M K(\mathbb{P}_n,\mathbb{P}_0)
\lesssim \tau^2  M \frac{  \omega^d  }{ M^{1+4/d}}.
\end{align}

\smallskip

\noindent {\bf Step 4}. \emph{Application of Fano's inequality}. 
By our choice of $M$ in \eqref{eq_M},
$$
\frac{  \omega^d  }{ M^{1+4/d}}\leq  \log \omega \lesssim \log M.$$
Combining this with \eqref{eq_c} we conclude that
\begin{align*}
\sum_{n=1}^M K(\mathbb{P}_n,\mathbb{P}_0)
\leq C \tau^2 M \log(M),
\end{align*}
for a constant $C>0$. We choose $\tau$ so that
$C \tau^2 < 1/8$ and, consequently, assumption (ii) in Proposition~\ref{thm:fano} is satisfied.
Since $N_\Omega \geq 3$, $\text{diam}(\Omega) > 1$ and
$\omega \asymp \text{diam}(\Omega)$. Note also that $N_\Omega^{1/d}\leq \text{diam}(\Omega)$. We now invoke Proposition ~\ref{thm:fano}, and combine its conclusion with the previous observations and \eqref{eq_M2},
\eqref{eq_x1} and \eqref{eq_x2} to estimate
\begin{align*}
\inf_{\widehat{S}_{X_{|\Omega}}}\sup_{S\in\mathcal{S}}\E\left\{\|\widehat{S}-S\|_\infty^2\right\}&\gtrsim   M^{-4/d}
\gtrsim \left(\frac{(\log \omega)^{1/d} }{\omega}\right)^{\frac{4d}{4+d}}
\gtrsim \left(\frac{(\log N_\Omega)^{1/d} }{\text{diam}(\Omega)}\right)^{\frac{4d}{4+d}},
\end{align*}
as desired.
\end{proof}

\bibliographystyle{abbrv}
\bibliography{paperbib}
\end{document}